\documentclass[11pt, a4paper]{article}
\usepackage{amsfonts, amsmath, amsthm, dsfont, longtable, url, enumitem, mathtools, tikz}
\usepackage{caption, subcaption}
\usetikzlibrary{arrows,decorations.markings}

\theoremstyle{plain}
\newtheorem{theorem}{Theorem}[section]

\newtheorem{corollary}[theorem]{Corollary}
\newtheorem{lemma}[theorem]{Lemma}
\newtheorem{proposition}[theorem]{Proposition}

\newtheorem*{claim*}{Claim}
\newtheorem*{problem*}{Problem}
\newtheorem*{conjecture*}{Conjecture}

\newtheorem{theoremIntro}{Theorem}

\theoremstyle{definition}
\newtheorem{definition}[theorem]{Definition}

\newtheorem{question}[theorem]{Question}

\newtheorem{definitionIntro}[theoremIntro]{Definition}

\newcommand\al{\alpha}
\newcommand\bt{\beta}
\newcommand\gm{\gamma}
\newcommand\dl{\delta}
\newcommand\lm{\lambda}

\newcommand\sg{\sigma}

\newcommand\cF{\mathcal{F}}

\newcommand\cM{\mathcal{M}}

\newcommand\la{\langle}
\newcommand\ra{\rangle}
\newcommand\lla{\langle\!\langle}
\newcommand\rra{\rangle\!\rangle}

\newcommand\ad{\mathrm{ad}}

\newcommand{\ch}{\mathrm{char}}

\newcommand\FF{\mathbb{F}}

\newcommand\Aut{\mathrm{Aut}}

\newcommand{\A}{\mathrm{A}}

\newcommand{\C}{\mathrm{C}}

\newcommand{\Miy}{\mathrm{Miy}}

\renewcommand{\phi}{\varphi}
\renewcommand{\epsilon}{\varepsilon}
\newcommand{\1}{\mathds{1}}

\setlist[enumerate,1]{label={\upshape (\arabic*)}}
\setlist[enumerate,2]{label={\upshape (\alph*)}}
\setlist[enumerate,3]{label={\upshape (\roman*)}}

\usepackage{xcolor}

\title{Split spin factor algebras}
\author{J.~M\textsuperscript{c}Inroy\footnote{School of Mathematics, University of Bristol, Fry Building, Woodland Road, Bristol, BS8 1UG, UK, and the Heilbronn Institute for Mathematical Research, Bristol, UK, email: justin.mcinroy@bristol.ac.uk}
 \and
 S.~Shpectorov\footnote{School of Mathematics, University of Birmingham, Edgbaston, Birmingham, B15 2TT, UK, email: s.shpectorov@bham.ac.uk}}

\date{}

\begin{document}
\maketitle

\begin{abstract}
Motivated by Yabe's classification of symmetric $2$-generated axial algebras of Monster type \cite{yabe}, we introduce a large class of algebras of Monster type $(\al, \frac{1}{2})$, generalising Yabe's $\mathrm{III}(\al,\frac{1}{2}, \dl)$ family. Our algebras bear a striking similarity with Jordan spin factor algebras with the difference being that we asymmetrically split the identity as a sum of two idempotents.  We investigate the properties of these algebras, including the existence of a Frobenius form and ideals.  In the $2$-generated case, where our algebra is isomorphic to one of Yabe's examples, we use our new viewpoint to identify the axet, that is, the closure of the two generating axes.
\end{abstract}

\section{Introduction}

The class of axial algebras was introduced by Hall, Shpectorov and Rehren \cite{Axial1, Axial2}.  Recently there has been much research into the class of algebras of Jordan type, which includes the classical Jordan algebras and also Matsuo algebras arising from $3$-transposition groups, but also into the wider class of axial algebras of Monster type $(\al,\bt)$ (as defined by the fusion law in Table \ref{tab:fusion}).  This class adds algebras for some sporadic finite simple groups including the Griess algebra for the Monster $M$.

Recently Yabe \cite{yabe} classified symmetric $2$-generated axial algebras of Monster type.  In doing so, he introduced several new families of $2$-generated algebras, in addition to those found by Rehren in \cite{gendihedral} and found by Joshi using the double axis construction \cite{JoshiMRes} (see also \cite{doubleMatsuo}).  The starting point of this article is an attempt to understand one of Yabe's families, $\mathrm{III}(\al, \frac{1}{2}, \dl)$.  These algebras have the puzzling property that while they are finite dimensional, they have potentially (depending on the field) infinitely many axes.  We were looking for a description of these algebras which exhibits their full symmetry and makes them easy to work with by hand.

We succeeded in doing so and, as a bonus, our description generalises to a much richer family with any number of generators.  Our description is also reminiscent of the spin factor Jordan algebras.  These are extensions of a quadratic space by an identity.  In our new family, we also start with a quadratic space, but instead expand by a $2$-dimensional piece with an identity asymmetrically split as the sum of two idempotents.  Because of this similarity, we call these algebras \emph{split spin factors}.

\begin{definitionIntro}
Let $E$ be a vector space with a symmetric bilinear form $b \colon E\times E \to \FF$ and $\al \in \FF$.  The \emph{split spin factor} $S(b,\al)$ is the algebra on $E \oplus \FF z_1  \oplus \FF z_2$ with multiplication given by
\[
\begin{gathered}
z_1^2 = z_1, \quad z_2^2 = z_2, \quad z_1 z_2 = 0, \\
e z_1 = \al e, \quad e z_2 = (1-\al) e, \\
ef = -b(e,f)z,
\end{gathered}
\]
for all $e,f \in E$, where $z := \al(\al-2) z_1 + (\al-1)(\al+1)z_2$.
\end{definitionIntro}

As can be seen, $z_1$ and $z_2$ are two idempotents and their sum $\1 = z_1 + z_2$ is the identity for the algebra.  If $\al = 1,0$, then $S(b, \al)$ is a direct sum of $\FF$ and the spin factor algebra coming from the quadratic space $E$.  Also, when $\al = \frac{1}{2}$, the algebra is isomorphic to the spin factor algebra for an extended quadratic space $\hat{E} = E \oplus \FF(z_1-z_2)$.  \emph{So we will assume in our statements that $\al \neq 1,0,\frac{1}{2}$.}

We classified all additional non-zero idempotents in this algebra and showed that they fall into two classes:
\begin{enumerate}
\item[\rm(a)] $\tfrac{1}{2}\left( e + \al z_1 + (\al+1) z_2 \right)$,
\item[\rm(b)] $\tfrac{1}{2}\left( e + (2-\al) z_1 + (1-\al) z_2 \right)$,
\end{enumerate}
where $e \in E$ of length $b(e,e) = 1$.

\begin{theoremIntro}
\begin{enumerate}
\item $z_1$ and $z_2$ are primitive axes of Jordan type $\al$ and $1-\al$, respectively.
\item Idempotents from families \textup{(a)} and \textup{(b)} are primitive axes of Monster type $(\al, \frac{1}{2})$ and $(1-\al, \frac{1}{2})$, respectively.
\end{enumerate}
\end{theoremIntro}

In particular, when $E$ is spanned by vectors of length $1$, the algebra is generated by axes and so it becomes an axial algebra of Monster type $(\al, \frac{1}{2})$, or $(1-\al, \frac{1}{2})$ depending on our choice.  In the former case, we can take family (a) with or without $z_1$ as axes, similarly in the second case, family (b) with or without $z_2$.  Note the symmetry between $\al$ and $\al' = 1-\al$.  In particular, $\al = 1-\al'$ and similarly for the coefficients of the vector $z$ and for the families (a) and (b).

The algebra admits a Frobenius form, that is, a non-zero symmetric bilinear form that associates with the algebra product.  In our case, this is just an extension of the form $b$.  The existence of the Frobenius form allows us to decide the simplicity of the algebra using the theory from \cite{axialstructure}.

\begin{theoremIntro}
$S(b,\al)$ is simple if and only if $b$ is non-degenerate and $\al\notin\{-1,2\}$.
\end{theoremIntro}

When $\al=-1,2$, the algebra is baric, that is, there is an ideal of codimension one, which is also the radical of the Frobenius form.  If $\al=-1$, then family (a) and $z_1$ are in the radical and the other way round for $\al=2$.

One of the main features of an axial algebra is that, when the fusion law is $C_2$-graded, we can associate to each axis $x$ an automorphism $\tau_x$ called the Miyamoto involution.  In our case, the Miyamoto involution associated to an axis from family (a), or (b) fixes $z_1$ and $z_2$ and acts on $E$ as $-r_e$, where $r_e$ is the reflection in the vector $e$.  The full automorphism group of the algebra is the orthogonal group $O(E, b)$.

The last part of the paper is about the $2$-generated case.  Here $E$ is necessarily $2$-dimensional spanned by vectors $e$ and $f$ of length $1$ and hence the form $b$ is fully determined by $\mu := b(e,f)$.

In view of the symmetry between $\al$ and $\al'=1-\al$, we can focus on two axes $x$ and $y$ from family (a).

\begin{theoremIntro}
Let $x=\frac{1}{2}(e+\al z_1+(\al+1)z_2)$ and $y=\frac{1}{2}(f+\al z_1+(\al+1)z_2)$.
If $\al\neq -1$ and $\mu\neq 1$ then $S(b,\al)=\lla x,y\rra$ is isomorphic to Yabe's algebra $\mathrm{III}(\al,\frac{1}{2},\dl)$ with $\dl=-2\mu-1$.
\end{theoremIntro}

We have already mentioned the case of $\al = -1$ above. This case is special: $z_1$ and family (a) do not generate the whole algebra, but a subalgebra $S(b,-1)^\circ$ of codimension $1$, which is in fact the radical of $S(b,-1)$. The subalgebra $S(b,-1)^\circ$ admits a nil cover $\widehat{S}(b,-1)^\circ$, which we also introduce and discuss in Section \ref{exception}. In the $2$-generated case, this cover is isomorphic to Yabe's $\mathrm{III}(-1,{1\over 2},\dl)$.  

The other exception in the $2$-generated case is $\mu = 1$. Here $xy = \frac{1}{2}(x+y)$ (both in $S(b,\al)$ and $\widehat{S}(b,-1)^\circ)$) and so $x$ and $y$ do not generate the whole algebra.  In the corresponding case of Yabe's algebra $\mathrm{III}(\al,\frac{1}{2}, -3)$, the identity of the algebra turns into a nil element (as it does for $\al=-1$).

Finally, we investigate the closure $X = x^D \cup y^D$ of axes $x$ and $y$ under the action of the Miyamoto group $D := \la \tau_x, \tau_y \ra$.  

In Section \ref{2gen}, we formulate precise statements for the size of $X$ depending on the value of $\mu$.  The outcome can be summarised here as follows: In characteristic $0$, $|X|$ is generically infinite, but can have any finite size $n \geq 2$ for specific values of $\mu$.  Whereas in characteristic $p$, $X$ can have infinite size if $\mu$ is transcendental over the prime subfield, size $p$ or $2p$ for $\mu=\pm 1$, or size coprime to $p$ otherwise.

\medskip
{\bf Acknowledgement:} The work of the second author has been supported by the
Mathematical Center in Akademgorodok under the agreement No. 075-15-2019-1675 
with the Ministry of Science and Higher Education of the Russian Federation.

\section{Axial algebras}

Throughout this paper we are considering commutative algebras that are not necessarily associative.

\begin{definition}
A \emph{fusion law} is a set $\cF$ with a binary operation $\ast \colon \cF\times\cF\to 2^\cF$, where $2^\cF$ is our notation for the set of all subsets of $\cF$. Just like we only consider commutative algebras, we will only consider \emph{commutative} fusion laws, i.e., we will have $\lm\ast\mu=\mu\ast\lm$ for all $\lm,\mu\in\cF$.
\end{definition}

In this paper we are concerned with the fusion laws shown in Table \ref{tab:fusion}.
\begin{table}[h!tb]
\[
\begin{array}{|c||ccc|c|}
\hline
\ast&1&0&\al&\bt\\
\hline\hline
1&1&&\al&\bt\\
0&&1&\al&\bt\\
\al&\al&\al&1,0&\bt\\
\hline
\bt&\bt&\bt&\bt&1,0,\al\\
\hline
\end{array}
\]
\caption{Fusion law of Monster type $(\al,\bt)$}\label{tab:fusion}
\end{table}
The full fusion law on $\{1,0,\al,\bt\} $ is called the \emph{fusion law of Monster type} and the sublaw on $\{1,0,\al\}$ is the \emph{fusion law of Jordan type}.

Suppose $A$ is an algebra defined over a field $\FF$.  For $a \in A$, let $\ad_a$ be the adjoint map and $A_\lm(a)$ be the $\lm$-eigenspace of $\ad_a$ (note that we allow this to be trivial).  For a set $N \subset \FF$, we write $A_N(a):=\bigoplus_{\nu\in N} A_\nu(a)$.

\begin{definition}
Suppose $1 \in \cF \subset \FF$ is a fusion law.  An \emph{$\cF$-axis} is a non-zero idempotent $a \in A$ such that
\begin{enumerate}
\item $A = \bigoplus_{\lm \in \cF} A_\lm(a)$
\item $A_\lm(a) A_\mu(a) \subseteq A_{\lm\ast\mu}(a)$, for all $\lm,\mu\in\cF$.
\end{enumerate}
We say that $a$ is \emph{primitive} if $A_1(a) = \la a \ra$.
\end{definition}

\begin{definition}
A commutative algebra $A$ together with a collection $X$ of $\cF$-axes which generate $A$, is called an \emph{$\cF$-axial algebra}.  The algebra is called \emph{primitive} if all axes in $X$ are primitive.
\end{definition}

In this paper, we are focussing on axial algebras of Jordan and Monster type, which are primitive algebras with the fusion laws in Table \ref{tab:fusion}.  Both these laws are $C_2$-graded.  This means that for every axis $a$, the algebra admits a $C_2$-grading $A = A_+(a) \oplus A_-(a)$, where $A_+(a) = A_1(a) \oplus A_0(a) \oplus A_{\al}(a)$ and $A_-(a) = A_\bt(a)$ for the Monster type fusion law and $A_+(a) = A_1(a) \oplus A_0(a)$ and $A_-(a) = A_\al(a)$ for the Jordan type fusion law.  Correspondingly, we have an involution $\tau_a \in \Aut(A)$, known as the \emph{Miyamoto involution} defined by
\[
v \mapsto
\begin{cases}
v & \mbox{if } v \in A_+(a) \\
-v & \mbox{if } v \in A_-(a)
\end{cases}
\]
and extended linearly to $A$.  The \emph{Miyamoto group} of $A$ is the subgroup $\Miy(X) \leq \Aut(A)$ generated by $\tau_a$ for all $a \in X$.

Often the algebra admits a bilinear form with a nice property.

\begin{definition}
A \emph{Frobenius form} on an axial algebra $A$ is a non-zero symmetric bilinear form $( \cdot, \cdot) \colon A \times A \to \FF$ such that
\[
(a,bc) = (ab,c)
\]
for all $a,b,c \in A$.
\end{definition}

\section{The algebra}

\begin{definition}
Suppose $E$ is a vector space over the field $\FF$ with a symmetric bilinear form $b \colon E \times E \to \FF$ and let $\al \in \FF$.  Define a commutative algebra $A = S(b,\al)$ on the vector space $E \oplus \FF z_1 \oplus \FF z_2$ by
\[
\begin{gathered}
z_1^2 = z_1, \quad z_2^2 = z_2, \quad z_1 z_2 = 0, \\
e z_1 = \al e, \quad e z_2 = (1-\al) e, \\
ef = -b(e,f)z,
\end{gathered}
\]
for all $e,f \in E$, where $z := \al(\al-2) z_1 + (\al-1)(\al+1)z_2$.  We call this algebra the \emph{split spin factor} algebra.
\end{definition}

Note that the algebra always has an identity given by $\1 = z_1 + z_2$. Note also the symmetry between $z_1$ and $z_2$: if we set $\al'=1-\al$ then $\al=1-\al'$ and hence $S(b,\al)=S(b,\al')$, where in the right side the r\^oles of $z_1$ and $z_2$ are switched.

Let us first dispose of the special cases $\al = 0, 1$.

\begin{proposition}\label{al10}
$S(b,0)$ and $S(b,1)$ are both the direct product of a spin factor Jordan algebra by a copy of the field.
\end{proposition}
\begin{proof}
We just do the case where $\al=0$.  In this case, $z = -z_2$ and so $B := E \oplus \FF z_2$ is the spin factor Jordan algebra.  Since $\al=0$, $z_1 b = 0$ for all $b \in B$, and so $A$ is the direct sum of $B$ with $\FF z_1 \cong \FF$.
\end{proof}

Another special situation is where $\al=\tfrac{1}{2}$, in which case, clearly, $\ch(\FF)\neq 2$. 

\begin{proposition}\label{al1/2}
$S(b,\tfrac{1}{2})$ is a spin factor Jordan algebra.
\end{proposition}

\begin{proof}
Since $\al=\tfrac{1}{2}$, we have that $z=-\tfrac{3}{4}z_1-\tfrac{3}{4}z_2=-\tfrac{3}{4}\1$. Let $u=z_1-z_2$ and note that $eu=ue=e(z_1-z_2)=\tfrac{1}{2}e-\tfrac{1}{2}e=0$ and $u^2=z_1^2+z_2^2=z_1+z_2=\1$. Set $\hat E=E\oplus\FF u$. Then for $v=e+\gm u$ and $w=f+\dl u$, where $e,f\in E$ and $\gm,\dl\in\FF$, we have that $vw=ef+\gm uf+\dl eu+\gm\dl u^2=(\tfrac{3}{4}b(e,f)+\gm\dl)\1$. The expression in the brackets is a symmetric bilinear form on $\hat E$. Hence, indeed, $S(b,\tfrac{1}{2})$ is a spin factor algebra.
\end{proof}

In particular, in all three special cases, $\al = 1,0,\frac{1}{2}$, the algebra is a Jordan algebra.  Note that we could also use the language of \cite{DPSV} and view these algebras as axial decomposition algebras for the full Monster type fusion law.

In view of Propositions \ref{al10} and \ref{al1/2}, \emph{in the remainder of the paper we will assume that $\al\neq 1,0, \tfrac{1}{2}$}.

Let us denote the subalgebra $\FF z_1 \oplus \FF z_2$ by $Z$; this is isomorphic to $\FF^2$. Let us also note that the orthogonal group $G=O(E,b)$, extended to the entire $A$ by letting all its elements fix $z_1$ and $z_2$, preserves the algebra product and hence is a subgroup of $\Aut(A)$. We will show later in the paper that $G$ is the full group of automorphisms of $A$.

The next step is to determine all idempotents in $A$.

\begin{proposition} \label{idempotents}
Let $A = S(b, \al)$ be the split spin factor algebra. Then a non-zero idempotent in $A$ is one of $\1$, $z_1$, $z_2$, or is in one of the following two families:
\begin{enumerate}
\item[\rm(a)] $\tfrac{1}{2}\left( e + \al z_1 + (\al+1) z_2 \right)$,
\item[\rm(b)] $\tfrac{1}{2}\left( e + (2-\al) z_1 + (1-\al) z_2 \right)$,
\end{enumerate}
where $e \in E$ such that $b(e,e) = 1$. Note that the two families require $\ch(\FF)\neq 2$.
\end{proposition}
\begin{proof}
Let $x = u + \gm z_1 + \dl z_2$ be a non-zero idempotent where $\gm, \dl \in \FF$. If $u=0$, then $x$ is an idempotent in $Z \cong \FF^2$ and so it is clearly one of the three idempotents $\1$, $z_1$, or $z_2$.  Now suppose that $u \neq 0$.  We deduce the relations:
\begin{align*}
x^2 &= (u + \gm z_1 + \dl z_2)^2 \\
&= -b(u,u)z +2\gm z_1 u + 2\dl z_2 u + 2 \gm \dl z_1z_2 + \gm^2 z_1 + \dl^2 z_2 \\
&= -b(u,u)\left(\al(\al-2) z_1 + (\al-1)(\al+1)z_2 \right) \\
&\phantom{{}={}} {} +2\gm \al u + 2\dl (1-\al) u + \gm^2 z_1 + \dl^2 z_2 \\
&= \left( 2\gm \al+ 2\dl (1-\al) \right) u \\
&\phantom{{}={}} {} +\left(\gm^2 - \al(\al-2)b(u,u) \right) z_1 \\
&\phantom{{}={}} {} +\left(\dl^2 - (\al-1)(\al+1)b(u,u) \right) z_2
\end{align*}
Since $x^2=x$ and $u \neq 0$, we have the following three equations
\begin{align}
1 &= 2\gm \al+ 2\dl (1-\al) \label{eqn1}\\
\gm^2 -\gm&= \al(\al-2)b(u,u) \label{eqn2}\\ 
\dl^2 -\dl &= (\al-1)(\al+1)b(u,u) \label{eqn3}
\end{align}
From Equation (\ref{eqn1}), we see that $\ch(\FF)\neq 2$ and $\al\gm = \tfrac{1}{2} - \dl(1-\al)$.  Hence,
\begin{align*}
\al^2(\gm^2-\gm) &= \al\gm(\al\gm-\al) \\
&= (\tfrac{1}{2} - \dl(1-\al))(\tfrac{1}{2}-\al - \dl(1-\al))) \\
&= \tfrac{1}{2}(\tfrac{1}{2}-\al) -(1-\al)(\tfrac{1}{2}-\al) \dl - \tfrac{1}{2}(1-\al)\dl + (1-\al)^2 \dl^2 \\
&=\tfrac{1}{2}(\tfrac{1}{2}-\al) + (1-\al)^2(\dl^2-\dl)
\end{align*}
So we have a system of linear equations for the variables $s = \gm^2-\gm$ and $t = \dl^2-\dl$ given by $\al^2 s -(1-\al)^2 t = \tfrac{1}{2}(\tfrac{1}{2}-\al)$ and $(\al^2-1)s - \al(\al-2)t = 0$.  The determinant of the matrix on the left hand side is $-\al^3(\al-2) + (1-\al)^2(\al^2-1) = 2\al-1\neq 0$ since $\al \neq \tfrac{1}{2}$. Therefore, we find that the unique solution is $s = \tfrac{1}{4}\al(\al-2)$ and $t = \tfrac{1}{4}(\al^2-1)$.  Substituting these into Equations (\ref{eqn2}) and (\ref{eqn3}) we get that $\tfrac{1}{4}\al(\al-2) = \al(\al-2)b(u,u)$ and $\tfrac{1}{4}(\al^2-1) = (\al^2-1)b(u,u)$.  Note that $\al(\al-2)$ and $\al^2-1$ are equal to zero at the same time if and only if $\al=-1=2$, which means that $\ch(\FF)=3$.  However, then $\al$ is also equal to $\tfrac{1}{2}$, a contradiction.  So, we see that $b(u,u)= \tfrac{1}{4}$.  Using this, we deduce from Equation (\ref{eqn2}) that $0=\gm^2 - \gm -\tfrac{1}{4}\al(\al-2) = (\gm - \tfrac{1}{2}\al)(\gm + \tfrac{1}{2}(\al-2))$ and so get the two solutions for $\gm$.  Then the corresponding values for $\dl$ come from Equation (\ref{eqn1}) and taking $e = 2u$ yields the result.
\end{proof}

To clarify the relationship between families (a) and (b), note that if $x=\tfrac{1}{2}\left( e + \al z_1 + (\al+1) z_2 \right)$ is in family (a) then $y=\1-x=z_1+z_2-x= \tfrac{1}{2}(-e+(2-\al)z_1+(1-\al)z_2)$ is in family (b), and vice versa.

Next we investigate which fusion law is satisfied for each non-zero, non-identity idempotent. We start with $z_1$ and $z_2$.

\begin{proposition} \label{z_i}
The idempotents $z_1$ and $z_2$ are primitive and they satisfy the fusion law of Jordan type $\al$ and $1-\al$, respectively.
\end{proposition}

\begin{proof}
In view of symmetry between $z_1$ and $z_2$, we just deal with $z_1$. We have that $z_1\in A_1(z_1)$, $z_2\in A_0(z_1)$, and $E\subseteq A_\al(z_1)$. Since $\{z_1\}\cup\{z_2\}\cup E$ contains a basis of $A$, we have equalities: $A_1(z_1)=\FF z_1$, $A_0(z_1)=\FF z_2$, and $A_\al(z_1)=E$. In particular, $z_1$ is primitive. Manifestly, $A_1(z_1)A_1(z_1)=A_1(z_1)$, $A_0(z_1)A_0(z_1)=A_0(z_1)$, and $A_1(z_1)A_0(z_1)=0$. Also, $(A_1(z_1)+A_0(z_1)) A_\al(z_1)\subseteq A_\al(z_1)$. Finally, $A_\al(z_1)A_\al(z_1)= EE\subseteq\FF z\subseteq A_1(z_1)+A_0(z_1)$.
\end{proof}

Since the Jordan type fusion law is $C_2$-graded, $A$ admits Miyamoto involutions $\tau_{z_i}\in\Aut(A)$. For $u=e+\gm z_1+\dl z_2\in A$, where $e\in E$ and $\gm,\dl\in\FF$, we set $u^-=-e+\gm z_1+\dl z_2$. From the description of the eigenspaces of $\ad_{z_1}$ above, we see that $\sg:=\tau_{z_1}$ is the central involution of $G=O(E,b)$ negating all of $E$. It follows also that $\sg=\tau_{z_2}$ and that $u^\sg=u^-$ for all $u\in A$. In particular, $u$ is an idempotent if and only if $u^-$ is an idempotent and they are in the same family.

We now turn to families (a) and (b) from Proposition \ref{idempotents}. In particular, in this segment of the paper $\ch(\FF)\neq 2$. Note again that the symmetry between $z_1$ and $z_2$ switches the r\^oles of families (a) and (b). Indeed, if we again set $\al'=1-\al$, then $\tfrac{1}{2}(e+(2-\al)z_1+(1-\al)z_2)=\tfrac{1}{2}(e+(\al'+1)z_1+ \al' z_2)$. This means that it suffices to consider an arbitrary idempotent $x=\tfrac{1}{2}(e+\al z_1+(\al+1)z_2)$ from family (a). Here $e\in E$ satisfies $b(e,e)=1$. 

We first observe the following.

\begin{proposition} \label{3C}
For an idempotent $x=\tfrac{1}{2}(e+\al z_1+(\al+1)z_2)$, the subspace $B_x:=\la x,x^-,z_1\ra$ is a subalgebra isomorphic to $3\C(\al)$.
\end{proposition}

\begin{proof}
We have that $xz_1=\tfrac{1}{2}(\al e+\al z_1)=\tfrac{\al}{2}(x+z_1-x^-)$ since $e=x-x^-$. Applying $\sg$, we also get $x^-z_1=\tfrac{\al}{2}(x^-+z_1-x)$. Finally, $xx^-=\tfrac{1}{4}(e+\al z_1+(\al+1)z_2)(-e+\al z_1+(\al+1)z_2)=\tfrac{1}{4}(-e^2+(\al z_1+(\al+1)z_2)^2)=\tfrac{1}{4}(z+\al^2z_1+(\al+1)^2z_2)=\tfrac{1}{4}(\al(\al-2)z_1+(\al-1)(\al+1)z_2+\al^2z_1+(\al+1)^2z_2)=\tfrac{1}{4}(2\al(\al-1)z_1+2\al(\al+1)z_2)=\tfrac{\al}{2}((\al-1)z_1+(\al+1)z_2)=\tfrac{\al}{2}(x+x^--z_1)$ since $x+x^-=\al z_1+(\al+1)z_2$.
\end{proof}

\begin{proposition} \label{families (a) and (b)}
Idempotents from families \textup{(a)} and \textup{(b)} from Proposition $\ref{idempotents}$ are primitive and satisfy the fusion law of Monster type $(\al,\tfrac{1}{2})$ and $(1-\al,\tfrac{1}{2})$, respectively.
\end{proposition}

\begin{proof}
As discussed before the proposition, it suffices to consider $x=\tfrac{1}{2}(e+\al z_1+(\al+1)z_2)$. We first determine the eigenspaces. Inside $B_x$, we see $\FF x\subseteq A_1(x)$, $\FF y\subseteq A_0(x)$, where $y=\1-x=\tfrac{1}{2}(-e+(2-\al) z_1+(1-\al) z_2)$, and $\FF(z_1-x^-)\subseteq A_\al(x)$. Next, if $f\in e^\perp$ (a hyperplane of $E$) then $xf=\tfrac{1}{2}(-b(e,f)z+\al^2 f+(1-\al)(\al+1) f)=\tfrac{1}{2}f$, since $b(e,f)=0$. This means that $A_{\frac{1}{2}}(x)\supseteq e^\perp$. Since, manifestly, $A=B_x\oplus e^\perp$, we conclude that $A_1(x)=\FF x$, $A_0(x)=\FF y$, $A_\al(x)=\FF(z_1-x^-)$, and $A_{\frac{1}{2}}(x)=e^\perp$. 

Turning to the fusion law for $x$, since $B_x$ is an algebra of Jordan type $\al$, we know all the fusion rules on the set $\{1,0,\al\}$. Note that the linear map $\tau \colon A\to A$ acting as identity on $B_x$ and negating $e^\perp$ coincides with $-r_e\in G$ ($r_e$ is the reflection in the hyperplane of $E$ perpendicular to $e$) and hence it is an automorphism of $A$ of order $2$. Hence, the fusion law of $x$ admits a $C_2$-grading with $A_+=B_x=A_1(x)+A_0(x)+A_\al(x)$ and $A_-=e^\perp=A_{\frac{1}{2}}(x)$. This readily implies that $x$ obeys the Monster fusion law of type $(\al,\frac{1}{2})$, as claimed.  
\end{proof}

We note for the future that $\tau_x=\tau=-r_e\in G$. This gives us the following fact.

\begin{proposition}\label{tauinv}
We have that $\tau_x=\tau_{x^-}=\tau_{\1-x}=\tau_{\1-x^-}$.  Moreover, these are the only axes with this Miyamoto involution.
\end{proposition}

Let us summarise.

\begin{theorem}\label{generation}
The algebra $A=S(b,\al)$, where $E\neq 0$, is a primitive axial algebra of Monster type $(\al,\tfrac{1}{2})$ if and only if $\ch(\FF)\neq 2$, $\al \neq -1$ and $E$ is spanned by vectors $e$ with $b(e,e)=1$.
\end{theorem}

\begin{proof}
If $\ch(\FF)=2$ then $z_1$ and $z_2$ are the only axes, and they only generate $Z$. Hence we assume that $\ch(\FF)\neq 2$. It is clear from the definition that, for every subspace $W\subseteq E$, $A_W:=W\oplus Z$ is closed for products and hence is a subalgebra. Let $W=\la e\in E\mid b(e,e)=1\ra$. If $W\neq E$ then according to Proposition \ref{idempotents}, $A_W$ contains all idempotents from $A$, and so $A$ is not generated by idempotents, so it is not an axial algebra. On the other hand, if $W=E$ and $\al \neq -1$ then it is immediate from Proposition \ref{idempotents} that $A$ is generated (in fact, spanned) by $z_1$ and the idempotents from family (a). (If $W=E$ and $\al = -1$, then $z_1$ and family (a) only span $E \oplus \FF z_1$ which is closed for products and hence is a subalgebra.)  By Propositions \ref{z_i} and \ref{families (a) and (b)}, all these idempotents are primitive and satisfy the fusion law of Monster type $(\al,\tfrac{1}{2})$.
\end{proof}

In view of symmetry between $z_1$ and $z_2$, $A$ is also an algebra of Monster type $(1-\al,\tfrac{1}{2})$ with respect to $z_2$ and family (b)---provided $\al \neq 1- (-1) = 2$. Also note that, in both realisations, $A$ is $1$-closed, that is, it is spanned by the axes. Finally, $A$ is also generated, under the same assumptions, by family (a) (respectively, (b)) alone. However, with this set of generating axes, $A$ is $2$-closed, since family (a) (respectively, (b)) spans a subspace of $A$ of codimension $1$ (we assume that $E\neq 0$).

At this point we are ready to determine the full automorphism group of $A=S(b,\al)$.

\begin{theorem}
Assume that $E \neq 0$.  Either $A\cong 3\C(\al)$, or $\Aut(A)=G=O(E,b)$.
\end{theorem}

\begin{proof}
Suppose that $A\not\cong 3\C(\al)$. We claim that every automorphism of $A$ fixes $z_1$ and $z_2$. Indeed, the adjoints of $z_1$ and $z_2$ have different spectrum, so it suffices to show that they are not conjugate to idempotents from families (a) and (b). If $\dim(E)\geq 2$ then the latter idempotents have $\tfrac{1}{2}$ in the spectrum, so they are not conjugate to $z_i$. If $\dim(E)=1$ and there exists $e\in E$ with $b(e,e)=1$ then $A\cong 3\C(\al)$ by Proposition \ref{3C}; a contradiction. In the remaining cases, families (a) and (b) are empty, and so the claim holds. We have shown that $\Aut(A)$ fixes $z_1$ and $z_2$.

Now, since $\Aut(A)$ fixes $z_1$, it stabilises every eigenspace of $\ad_{z_1}$; in particular, $E=A_\al(z_1)$ is left invariant under $\Aut(A)$. Since $\Aut(A)$ acts as the identity on $Z=\la z_1,z_2\ra$ and since $A=E\oplus Z$, it follows that $\Aut(A)$ acts faithfully on $E$. Furthermore, since $b(e,f)z=(b(e,f)z)^\phi=(-ef)^\phi=-e^\phi f^\phi=b(e^\phi,f^\phi)z$ for all $e,f\in E$ and $\phi\in\Aut(A)$, we see that $\Aut(A)$ preserves the form $b$ and this means that $\Aut(A)=G=O(E,b)$, as claimed.
\end{proof}

The automorphism group of $A=3\C(\al)$, $\al\neq\tfrac{1}{2}$, is known to be isomorphic to $S_3$, and it is strictly bigger than $O(E,b)$, which is of order $2$ in this case. So here we have a true exception to the theorem.

Recall that a Frobenius form on an algebra $A$ is a non-zero symmetric bilinear form that associates with the algebra product.

\begin{theorem} \label{Frobenius form}
The algebra $A=S(b,\al)$ admits a Frobenius form $( \cdot, \cdot)$ given by
\[
\begin{gathered}
(e,f) = (\al+1)(2-\al)b(e,f), \quad (e,z_1)=(e,z_2) = 0, \\
(z_1,z_1) = \al+1, \quad (z_2,z_2) = 2-\al, \quad (z_1,z_2)=0,
\end{gathered}
\]
for all $e,f \in E$ and extended linearly to $A$.
\end{theorem}
\begin{proof}
We begin by noting that the form is invariant under the symmetry which exchanges $z_1$ and $z_2$ and also exchanges $\al$ and $\al' = 1-\al$.  In light of this, we just need to check $(a,bc) = (ab,c)$ for the following triples $(a,b,c)$: $(z_1,z_1,z_1)$, $(z_1,z_1,z_2)$, $(z_1,z_1,e)$, $(z_1,e,f)$ and $(e,f,g)$, for all $e,f,g \in E$.  Since the form is symmetric, $(z_1^2,z_1) = (z_1,z_1^2)$.  Notice that $z_2$ and $e$ are a $0$- and $\al$-eigenvector for $z_1$, respectively.  Since $(z_1,z_2)=0=(z_1,e)$, we get $(z_1^2,z_2) = 0 = (z_1,z_1z_2)$ and $(z_1^2,e) = 0 = (z_1,z_1e)$.  For the remaining two, we calculate: $(z_1, ef) = \left(z_1, -b(e,f)(\al(\al-2)z_1+(\al^2-1)z_2)\right) = -\al(\al+1)(\al-2)b(e,f)$ and $(z_1e,f) = \al(e,f) = \al(\al+1)(2-\al)b(e,f)$ which are equal.  Finally, $(e,fg) = \left(e,-b(f,g)z\right) = 0$ and hence $(e,fg) = 0 = (ef,g)$, as required.
\end{proof}

Since the Frobenius form is an extension of (a scaled version of) $b$, it is invariant under $G = O(E,b)$.  One can also check that if $\dim(E) = 1$ and $A \cong 3\A$, then the Frobenius form is invariant under $S_3$.

Note that for the particular scaling of the form that we have chosen, idempotents in family (a) have length $\al+1$, the same as $z_1$, and those in family (b) have length $2-\al$, the same as $z_2$.

\section{Simplicity}

In this section we additionally assume that $\ch(\FF)\neq 2$, and $E$ is spanned by vectors $e$ with $b(e,e)=1$.

Recall that the algebra $S(b,\al)$ admits a Frobenius form. We first discuss the case where some of the axes have length $0$ with respect to the form. Recall also that $z_1$ and the idempotents from family (a) have length $\al+1$ while $z_2$ and the idempotents from family (b) have length $2-\al$. Hence the special cases are $\al=-1$ and $\al=2$. We note that $-1=2$ only when $\ch(\FF)=3$, in which case $\al$ is also equal to $\frac{1}{2}$, which we assumed not to be the case. Thus, at least half of the idempotents are always non-singular.

From \cite{axialstructure}, the radical of the Frobenius form is an ideal.

\begin{proposition}
If $\al=-1$ then the Frobenius form has rank $1$ and its radical coincides with $E\oplus\FF z_1$. Symmetrically, if $\al=2$ then the Frobenius form also has rank $1$ and its radical coincides with $E\oplus\FF z_2$.
\end{proposition}

\begin{proof} 
Because of the definition of the Frobenius form, $E$ is in the radical for $\al=-1$ and $2$. Similarly, $z_1$ is in the radical for $\al=-1$ and $z_2$ is in the radical for $\al=2$.
\end{proof}

Note that in these cases $S(b,\al)$ is baric. We will explore this in more detail in the next section.

For the remainder of the section, we assume that $\al\neq -1, 2$. Then all non-zero, non-identity idempotents are non-singular. Recall that, according to \cite{axialstructure}, ideals split into two kinds: the ones that do not contain axes, and the ones that do.

Ideals of the first kind are contained in the radical of the algebra, which is defined as the largest ideal not containing axes, and which in our case coincides with the radical of the Frobenius form (see \cite[Theorem 4.9]{axialstructure}). 

\begin{proposition}
If $\al\notin\{-1,2\}$, the radical of $S(b,\al)$ coincides with the radical of the form $b$.
\end{proposition}

\begin{proof}
Since $z_1$ and $z_2$ are non-singular and $\FF z_1$ and $\FF z_2$ split off as direct summands, the radical of the Frobenius form is contained in $E$ and hence the claim holds.
\end{proof}

The ideals of the second kind are controlled by the projection graph, which in the present situation is the same as the non-orthogonality graph on the set of axes (cf. \cite[Lemma 4.17]{axialstructure}). 

\begin{proposition}
If $\al\notin\{-1,2\}$ then there are no proper ideals in $S(b,\al)$ that contain axes.
\end{proposition}

\begin{proof}
Without loss of generality, we view $A$ as an algebra of type $(\al,\frac{1}{2})$, that is, we can assume that the set of axes is the union of family (a) with $\{z_1\}$. Now from the definition of the Frobenius form and from the description of the family (a), it is clear that $z_1$ is non-orthogonal to all idempotents in family (a), which means that the non-orthogonality graph is connected. According to \cite[Corollary 4.15]{axialstructure}, this means that none of these idempotents lie in a proper ideal.
\end{proof}

We can now summarise when $S(b,\al)$ is simple.

\begin{theorem}
The algebra $S(b,\al)$ is simple if and only if $b$ is non-degenerate and $\al\notin\{-1,2\}$.
\end{theorem}

Let us finish this section with the following question concerning the baric situation.

\begin{question}
When $A = \lla x,y \rra$ and $\al = 2$, is the baric algebra $A$ a quotient of the highwater algebra $\cal{H}$ \cite{highwater}.
\end{question}

We expect that it is a quotient of $\cal{H}$, at least for some values of $\mu$.

\section{$2$-generated case}\label{2gen}

We write $\lla x, y \rra$ for the subalgebra generated by $x$ and $y$.  In this section, we assume that $A = \lla x,y \rra$ is generated by two axes $x=\frac{1}{2}(e+\al z_1+(\al+1)z_2)$ and $y=\frac{1}{2}(f+\al z_1+(\al+1)z_2)$ from family (a).  Equivalently, $E=\la e,f\ra$, where $b(e,e)=1=b(f,f)$. We let $\mu:=b(e,f)$.  

\begin{theorem}
If $\al\neq -1$ and $\mu\neq 1$ then $S(b,\al)=\lla x,y\rra$ is isomorphic to Yabe's algebra $\mathrm{III}(\al,\frac{1}{2},\dl)$ with $\dl=-2\mu-1$.
\end{theorem}

\begin{proof}
Let $a_0=x$, $a_1=y$, and $a_{-1}=y^{\tau_x}=\frac{1}{2}(-f+2\frac{b(e,f)}{b(e,e)}e+\al z_1+(\al+1)z_2)=\frac{1}{2}(-f+2\mu e+\al z_1+(\al+1) z_2)$. Assuming that $\al\neq -1$ and $\mu\neq 1$, let $q=\frac{\al(\al+1)(\mu-1)}{4}\1$. We verified in MAGMA that these four elements satisfy the relations for the basis of $\mathrm{III}(\al,\frac{1}{2},-2\mu-1)$ and $S(b,\al)=\lla x,y\rra=\la x,y,a_{-1},q\ra$.
\end{proof}

We have already discussed what happens when $\al=-1$. Namely, $x$ and $y$ are in the radical of the algebra and so they do not generate it. If $\mu=1$ then $x$ and $y$ also do not generate $S(b,\al)$. In fact, in this case, $xy=\frac{1}{2}(x+y)$ and $\lla x,y\rra=\la x,y\ra$ is a $2$-dimensional Jordan algebra.

In the remainder of this section, we investigate the \emph{gonality} of the algebra $S(b,\al)=\lla x,y\rra$, that is, the cardinality of $X:=x^D\cup y^D$\footnote{In \cite{axet}, we introduce the notion of an axet $(D,X,\tau)$ encoding the action of the Miyamoto group $D$ on the closed set of axes $X$ together with the $\tau$-map.  So, in a sense, in the remainder of this section we identify the axet arising in $A$.}, where $D=\la\tau_x,\tau_y\ra$ is the Miyamoto group of the $2$-generated algebra. Note that $X$ may be a small part of all the available axes, just as $D$ may be a proper subgroup of $\Aut(A)$. Since $D$ fixes $z_1$ and $z_2$, it acts faithfully on $E$.

It will actually be convenient to work with a larger dihedral group. Let $\theta$ be the automorphism of $A$ switching $x$ and $y$, and let $\hat D=\la\theta,\tau_x\ra$. Clearly, $\tau_x^\theta=\tau_y$ and so $\hat D$ contains $D$ as a subgroup of index at most $2$. Note that $\hat D$ acts transitively on $X$ and, furthermore, 
the stabilizer of $x$ (even in the whole $\Aut(A)\cong O(E,b)$) coincides with $\la\tau_x\ra=\la -r_e\ra$. Therefore, $|X|=\frac{1}{2}|\hat D|=|\rho|$, where $\rho:=\theta\tau_x$.  

\begin{lemma} \label{action}
The action of $\rho$ on $E$ is given by
\[
\rho = \begin{pmatrix} 2\mu & -1 \\ 1 & 0 \end{pmatrix}
\]
\end{lemma}
\begin{proof}
Since $\theta$ switches the basis vectors $e$ and $f$, and $\tau_x = -r_e$, the result follows from a calculation.
\end{proof}

Note that since the determinant of $\rho$ is $1$, $\rho$ have eigenvalues $\zeta$ and $\zeta^{-1}$ in (a suitable extension of) $\FF$.  Also, $\zeta + \zeta^{-1} = \mathrm{tr}(\rho) = 2\mu$ and hence $\mu = \frac{\zeta+\zeta^{-1}}{2}$.

\begin{lemma}\label{axet}
\begin{enumerate}
\item If $\zeta$ is not a root of unity, then $|X|$ is infinite.
\item If $\zeta \neq \pm 1$ is of order $n$, then $|X|=n$.
\item Suppose that $\zeta = \pm 1$. If $\FF$ has characteristic $0$ then $|X|=\infty$; if $\FF$ has characteristic $p>0$ then $|X|=p$ if $\zeta=1$ and $2p$ if $\zeta=-1$.
\end{enumerate}
\end{lemma}
\begin{proof}
We have seen that $|X|=|\rho|$. If $\zeta\neq \pm 1$ then $\zeta\neq\zeta^{-1}$, making $\rho$ diagonalisable. Hence $|\rho|=|\zeta|$, giving the first two claims. If $\zeta= \pm 1$ then the above matrix for $\rho$ is not diagonal, which means that its Jordan normal form is the $2 \times 2$ block $\begin{psmallmatrix} \zeta & 1 \\ 0 & \zeta \end{psmallmatrix}$.  Hence, the last claim follows.
\end{proof}

\begin{corollary} \label{closed}
Suppose that $\FF$ is algebraically closed.  In characteristic $0$, the size of $X$ can be any finite $n$ as well as infinity.  In positive characteristic $p$, the size is finite as long as $\zeta$ is algebraic over $\FF_p$ and it can be $p$, $2p$, or any number $n$ coprime to $p$.
\end{corollary}

For a non algebraically closed field, the best we can say is that in positive characteristic $p$, if the size of $X$ is finite, it is either $p$, $2p$, or coprime to $p$.

We have already mentioned that the index of $D$ in $\hat D$ can be $1$ or $2$. We remark that the former possibility occurs exactly when $|X|=|\rho|$ is odd. Indeed, $D=\la \tau_x,\tau_y\ra=\la \tau_x,\tau_x^\theta\ra$ and this is an index $2$ subgroup of the dihedral group $\hat D=\la\theta,\tau_x\ra$, unless $|\hat D|$ is twice odd. 

Also note the $D$-orbit structure on $X$: if $|X|$ is even (or infinite) then $x^D$ and $y^D$ are disjoint and have length $\frac{1}{2}|X|$. If $|X|$ is odd then $x^D=y^D=X$. This also follows directly from the properties of the dihedral group.

\section{Exceptional algebra} \label{exception}

We have shown in Theorem \ref{generation} that $S(b, \al)$ is generated by $z_1$ and family (a) provided that $\al \neq -1$.  In this section we construct an additional algebra that arises in the case where $\al = -1$.  We have already mentioned that in this case, $z_1$ and family (a) are contained in the subalgebra $S(b, -1)^\circ := E \oplus \FF z_1$ and the latter is an algebra of Monster type $(-1, {1 \over 2})$ provided $E$ is spanned by vectors of norm $1$.  The algebra which we introduce is a cover of this.

\begin{definition}
Let $E$ be a vector space with a symmetric bilinear form $b$.  Let $A=\widehat{S}(b,-1)^\circ$ be the algebra on $E \oplus \FF z_1 \oplus \FF n$ with multiplication
\[
\begin{gathered}
z_1^2 = z_1, \quad n^2 = 0, \quad z_1 n = 0, \\
e z_1 =  -e, \quad e n = 0, \\
ef = -b(e,f)z,
\end{gathered}
\]
for all $e,f \in E$, where $z := 3 z_1 -2n$.
\end{definition}

It is clear that $\la n \ra$ is a nil ideal and $\widehat{S}(b,-1)^\circ/\la n \ra \cong S(b, -1)^\circ$, so our new algebra is indeed a cover\footnote{Note that the construction of this cover is an example of a more general construction considered in \cite{expansions}.}.  Also, note that since $n$ annihilates the entire algebra, $A=\widehat{S}(b,-1)^\circ$ has no identity.

The properties of $A$ are established in a very similar way to that for $S(b, \al)$.  We will only indicate places where the calculations are slightly different.  For $S(b, \al)$, taking $\al = \frac{1}{2}$ gave a special case where the algebra was a Jordan algebra and so we excluded that.  Similarly here, if $-1={1\over 2}$ (i.e., if $\ch(\FF)=3$), the above algebra is a Jordan algebra (directly checking the Jordan identity). In view of this, for the remainder of the section, we assume that $\ch(\FF) \neq 3$.

\begin{proposition}
A non-zero idempotent in $\widehat{S}(b,-1)^\circ$ is either $z_1$, or an element from the family
\[
\tfrac{1}{2}(e -z_1 +n)
\]
for $e \in E$ such that $b(e,e)=1$.  Again the family only exists when $\ch(\FF) \neq 2$.
\end{proposition}
\begin{proof}
Let $x = u + \gm z_1 + \dl n$ be a non-zero idempotent where $\gm,\dl \in \FF$.  If $e=0$, then $x^2 = \gm^2 z_1$ and hence we obtain $z_1$.  Assuming that $u \neq 0$, we get
\begin{align*}
x^2 &= -b(u,u)(3 z_1 -2n) + \gm^2 z_1 - 2\gm u \\
&= -2 \gm u + (\gm^2 -3b(u,u)) z_1 +2b(u,u) n
\end{align*}
Since this is equal to $x$ and $u \neq 0$, we get the equations $1 = -2 \gm$, $\gm = \gm^2-3b(u,u)$ and $\dl = 2b(u,u)$.  From the first equation we get $\gm = -{1 \over 2}$ and substituting this into the second, we get $b(u,u) = {1 \over 4}$ as $\ch(\FF) \neq 3$.  Finally the last equation gives $\dl = {1 \over 2}$.  Taking $e = 2u$ yields the result.
\end{proof}

Similarly to before, $z_1$ is a primitive axis of Jordan type $-1$ and its Miyamoto involution $\sg$ negates vectors in $E$.  Hence it switches every idempotent $x = {1 \over 2}(e -z_1 +n)$ with $x^- := {1 \over 2}(-e -z_1 +n)$.

\begin{proposition} \label{3C and monster}
For an idempotent $x=\tfrac{1}{2}(e -z_1+n)$, the subspace $B_x:=\la x,x^-,z_1\ra$ is a subalgebra isomorphic to $3\C(-1)$.  Consequently, $x$ is a primitive axis of Monster type $\cM(-1, {1 \over 2})$.
\end{proposition}

The proof of this statement is very similar to that for Propositions \ref{3C} and \ref{families (a) and (b)} combined and so we omit it. Note that $B_x:=\la x,x^-,z_1\ra = \la e, z_1, n \ra$ and, more generally, for every subspace $W \subset E$, $W \oplus \FF z_1 \oplus \FF n$ is a subalgebra isomorphic to $\widehat{S}(b|_W,-1)^\circ$.

\begin{corollary}
$\widehat{S}(b,-1)^\circ$ is an axial algebra of Monster type $(-1, \frac{1}{2})$ if and only if $\ch(\FF)\neq2$ or $3$ and $E$ is spanned by vectors of norm $1$.
\end{corollary}

Turning to automorphisms, just like for the algebra $S(b, \al)$, we have $\tau_x = -r_e$ and $\tau_x = \tau_{x^-}$.

\begin{theorem}
Suppose that $\ch(\FF)\neq 2$ or $3$. Let $A = \widehat{S}(b,-1)^\circ$ and assume that $b \neq 0$.  Then either $A \cong 3\C(-1)$, or $\Aut(A) = G = O(E,b)$.
\end{theorem}
\begin{proof}
Suppose that $A \not\cong 3\C(-1)$.  This means that either $z_1$ is the only idempotent (if $E$ contains no vectors of norm $1$), or all other idempotents have a non-trivial ${1 \over 2}$-eigenspace.  Indeed, if $z_1$ is not the only idempotent and since $A \not\cong 3\C(-1)$, then $\dim(E) \geq 2$ and so $e^\perp$ is non-empty.  Hence the adjoint of $z_1$ has a unique spectrum among all idempotents and so $z_1$ is fixed by the automorphism group $G$.  Therefore, the eigenspaces $A_{1 \over 2}(z_1) = E$ and $A_0(z_1) = \la n \ra$ are both invariant under $G$.

Furthermore, we claim that the form $b$ on $E$ is preserved by $G$.  Indeed, $(uv)z_1 = -3b(u,v) z_1$ and so for $g \in \Aut(A)$, $-3b(u,v)z_1 = ((uv)z_1)^g = (u^gv^g)z_1^g = -3b(u^g,v^g) z_1$.  Since $\ch(\FF) \neq 3$, $g$ preserves the form $b$.

Finally, we need to show that $G$ fixes $n$. Since $b\neq 0$, we have $e,f\in E$ such that 
$b(e,f)\neq 0$. Then, for $g\in G$, we have that $b(e,f)z^g=(ef)^g=e^gf^d=b(e^g,f^g)z=
b(e,f)z$. Since $b(e,f)\neq 0$, we have that $z^g=z$, that is, $(3z_1-2n)^g=3z_1-2n$. 
Since $z_1^g=z_1$ and $\ch(\FF)\neq 2$, we conclude that $n^g=n$.
\end{proof}

Let us add here that assumptions that $A\not\cong 3\C(-1)$ and $b\neq 0$ are both necessary. Indeed, if $A=3\C(-1)$ then $G\cong S_3>O(E,b)$. If $b=0$ then $G$ also contains automorphisms that fix $E\oplus\FF z_1$ and acts on $\FF n$ as an arbitrary scalar, so $G$ is larger than $O(E,b)$ in this case, too.

\begin{theorem}
$\widehat{S}(b,-1)^\circ$ admits a Frobenius form given by
\[
\begin{gathered}
(e,f) = 3b(e,f), \quad (e,z_1)=(e,n) = 0, \\
(z_1,z_1) = 1, \quad (n,n) = (z_1,n)=0,
\end{gathered}
\]
for all $e,f \in E$ and extended linearly to $A$.
\end{theorem}

We skip the proof, which is quite similar to that of Theorem \ref{Frobenius form}.

The algebra $A$ is never simple, since $\la n\ra$ is an ideal. However, we can formulate the following, where we assume that $E$ is spanned by vectors of norm $1$, i.e., $A$ is an algebra of Monster type $(-1,{1\over 2})$.

\begin{theorem}
Assuming that $\ch(\FF)\neq 2,3$, every proper ideal of $\widehat{S}(b,-1)^\circ$ is contained in the radical, which coincides with $E^\perp\oplus\la n\ra$, where $E^\perp=\{e\in E:b(e,f)=0, \mbox{ for all }f\in E\}$.
\end{theorem}

In the $2$-generated case, the algebra $A$ gives us the exceptional case of Yabe's algebra 
$\mathrm{III}(-1,{1\over 2},\dl)$.

\begin{theorem}
Suppose that $E=\la e,f\ra$ is $2$-dimensional with $(e,e)=(f,f)=1$ and $(e,f)=\mu$. 
Then $\widehat{S}(b,-1)^\circ\cong\mathrm{III}(-1,{1\over 2},-2\mu-1)$, provided that $\mu\neq 1$.
\end{theorem}

By computation in MAGMA, the required isomorphism is given by: $a_0=x={1\over 2}(e-z_1+n)$, $a_1=y={1\over 2}(f-z_1+n)$, $a_{-1}=y^{\tau_x}={1\over 2}(-f+2\mu e-z_1+n)$, and $q={1-\mu\over 4}n$. Note that in this theorem $\ch(\FF)\neq 2,3$, since Yabe's algebra is not defined in these characteristics.

Finally, let us discuss the gonality of the $2$-generated algebra $A$. Here the situation is identical to that of Section \ref{2gen}. Namely, since the action of $\tau_x$ and $\tau_y$ on $E$ is again given by negative reflections, Lemmas \ref{action}, \ref{axet} and Corollary \ref{closed} apply unchanged.

\end{document}